\documentclass[12pt]{amsart}
\usepackage{amssymb}
\usepackage{amsfonts}
\usepackage{latexsym}
\usepackage{amsmath}

\usepackage{color}

\usepackage{graphicx}

\newtheorem{thm}{Theorem}[section]

\newtheorem{corollary}[thm]{Corollary}
\newtheorem{lemma}[thm]{Lemma}

\newtheorem{theorem}[thm]{Theorem}

\newtheorem{claim}[thm]{Claim}

\newtheorem{remark}[thm]{Remark}

\numberwithin{equation}{section}

\newcommand{\bbC}{{\Bbb C}}
\newcommand{\CC}{{\widehat{\bbC}}}
\newcommand{\C}{{\mathbb C}}

\newcommand{\R}{{\mathbb R}}

\newcommand{\N}{{\mathbb N}}

\begin{document}

\title[Random Backward Iteration Algorithm for Julia sets]{Random Backward Iteration Algorithm for Julia sets of Rational Semigroups}

\author{Rich Stankewitz}
\address[Rich Stankewitz]
    {Department of Mathematical Sciences\\
    Ball State University\\
    http://rstankewitz.iweb.bsu.edu/}
\email{rstankewitz@bsu.edu}

\author{Hiroki Sumi}
\address[Hiroki Sumi]
    {Department of Mathematics\\
    Graduate School of Science\\
    Osaka University\\
    1-1, Machikaneyama, Toyonaka\\
    Osaka, 560-0043, Japan\\
    http://www.math.sci.osaka-u.ac.jp/$\sim $sumi/welcomeou-e.html}
\email{sumi@math.sci.osaka-u.ac.jp}

\begin{abstract}
We provide proof that a random backward iteration algorithm, previously proven to work in the context of iteration of a rational function of degree two or more, extends to rational semigroups (of a certain type).  We also provide some consequences of this result.
\end{abstract}

\date{ \today }

\maketitle

\section{Introduction}

The work of Boyd~\cite{Boyd2}, later generalized by Sumi~\cite{Su3}, provides a formal justification for a method we call the ``full backward iteration algorithm" for generating graphical computer approximations of Julia sets  of finitely generated rational semigroups where each map is of degree two or more.  This method has been used by the present authors and many others to generate such fractal images.  However, it was also noted that the ``random backward iteration algorithm"  (also known as the ``ergodic method" or "chaos game" method) seemed to always provide the same pictures, though formal proof had not yet been given, except in specialized examples (see~\cite{Fujishima}).  The goal of this paper is to provide a formal proof of the strict mathematical sense in which the random backward iteration algorithm works.  Analogous results, for both the ``full" and ``random" methods, in the context of attractor sets for contracting iterated function systems (see~\cite{Hutchinson, BEH, Elton, RSThesis}) and Julia sets of a iterated single rational functions (see~\cite{HawkinsTaylor, Lyu, Mane, FLM}) are known.  In these mentioned works the benefits of the random (backward) iteration method over the full (backward) iteration method in terms of simplicity, speed, and memory are described and so we will not repeat these same benefits here.  Both the full and random methods for both attractor sets of iterated function systems and Julia sets of rational semigroups are implemented using the freely available application Julia 2.0~\cite{Julia2.0}.

The Julia set (defined below) is the set of points where there is chaos or instability in the dynamics, and as such is of great interest in many settings.  As described in~\cite{SumiRandom, SumiCooperation} the setting of rational semigroups is of real importance in the study of, say, population dynamics of a species which, unlike the iteration case, has multiple choices of survival strategies to employ from year to year.  Methods for drawing Julia sets for the more general setting of rational semigroups are more limited than those available for iteration.  For example, the iteration methods which are based on the presence of an attracting fixed point (e.g., using an ``escape to infinity criterion" for polynomials $z^2+c$) require the complete invariance of the Julia and Fatou sets and are thus not adaptable to semigroups where these sets only have partial invariance.  It is, however, interesting to note that in the context of polynomial semigroup dynamics, studying the probability $P(z)$ that a random forward obit starting at $z$ ``escapes" to infinity can produce quite beautiful pictures and intriguing mathematics (see~\cite{SumiRandom, SumiCooperation}). In fact, the function $P(z)$ is often a continuous function on all of $\C$, but is almost everywhere flat, varying only on the thin fractal Julia set of the related semigroup.  Hence, extending to the semigroup setting the backward iteration method is important as it adds to the number of reliable ways one can now computer graphically illustrate this important locus of instability, i.e., the Julia set.

We now introduce the basic terminology and notions needed to describe the full and random methods mentioned above.

Let $G=\langle f_1, \dots, f_k \rangle$ be a rational semigroup generated by non-constant rational maps $f_j$ where the semigroup operation is composition of functions, i.e., $G$ is the collections of all maps which can be expressed as a finite composition of maps from the generating set $\{f_j:j=1, \dots, k\}$.

Research on the dynamics of rational
semigroups was initiated by Hinkkanen and Martin
in~\cite{HM1}, where each rational semigroup was always taken to have at least one element of degree at least two.  Though the semigroups $G$ to which our main theorem apply will meet this criteria, we also have need to define the relevant notions of Fatou and Julia set for semigroups of M\"obius maps.  (For an in depth look at the dynamics of M\"obius semigroups see~\cite{FMS}.)  Also, Ren, Gong, and
Zhou studied rational semigroups from the perspective of random
dynamical systems (see~\cite{ZR,GR}).

We thus follow~\cite{HM1} in saying that for a rational semigroup $G$ the \textit{Fatou set} $F(G)$ is the set of points in $\CC$ which have a neighborhood on which $G$ is normal, and its complement in $\CC$ is called the \textit{Julia set} $J(G)$.  The more classical Fatou
set and Julia set of the cyclic semigroup $\langle g \rangle$ generated by a single map (i.e., the collection of iterates $\{g^n: n \geq 1\}$) is
denoted by $F(g)$ and $J(g)$, respectively.

We quote the following results from~\cite{HM1}. The Fatou set $F(G)$
is \textit{forward invariant} under each element of $G$, i.e.,
$g(F(G)) \subseteq F(G)$ for all $g \in G$, and thus $J(G)$ is
\textit{backward invariant} under each element of $G$, i.e.,
$g^{-1}(J(G)) \subseteq J(G)$ for all $g \in G$.

We should take a moment to note that the sets $F(G)$ and $J(G)$ are,
however, not necessarily completely invariant under the elements of
$G$. This is in contrast to the case of \textit{iteration} dynamics,
i.e., the dynamics of \textit{cyclic} semigroups generated by a single function. For a treatment of alternatively defined
\textit{completely} invariant Julia sets of rational semigroups the
reader is referred to~\cite{RSThesis, RS1, RS2, SSS}.

\begin{remark}\label{AssumptionsG}
For the rest of this paper we assume that the rational semigroup $G=\langle f_1, \dots, f_k \rangle$ satisfies the following conditions:
\begin{enumerate}
  \item The semigroup $G$ contains at least one element of degree two or more.
  \item The exceptional set $E(G)$, set of points $z$ with a finite backward orbit $\cup_{g \in G} g^{-1}(\{z\})$, is contained in $F(G)$.  Note that~\cite{HM1} shows $\textrm{card}(E(G)) \leq 2$.
  \item The M\"obius semigroup $H=\{ h^{-1}: h\in G \mbox{ is Mobius}\} $ satisfies $F(H) \supseteq J(G)$ (setting $F(H)=\CC$ when $H=\emptyset$).
\end{enumerate}
\end{remark}

Note that~\cite{Boyd2} requires the stronger assumption that $G=\langle f_1, \dots, f_k \rangle$ contain no M\"obius maps, from which all three of the above assumptions follow.  These relaxed assumptions on $G$ are used in~\cite{Su3} to show the following, whose proof we include here for completeness.

\begin{lemma}\label{LemmaK}
Letting $G$ be as in Remark~\ref{AssumptionsG} and given any $a \in \CC \setminus E(G)$, there exists a compact set $K \subseteq \CC \setminus E(G)$ that is backward invariant under $G$ and that contains $a$.
\end{lemma}

\begin{proof}
If $E(G)= \emptyset$, then $K=\CC$ will suffice.  Now suppose $\textrm{card}(E(G))=2$.  Since $E(G)$ is a finite backward invariant set under $G$, it is also forward invariant under $G$.  Since $\textrm{card}(J(G)) \geq 3$ (which follows since $J(f) \subseteq J(G)$ for all $f \in G$ and some $f$ has degree two or more), each component of $F(G)$ supports a hyperbolic metric.  Let $D$ be the union of two disjoint hyperbolic disks of radius $\delta$ each centered at a point in $E(G)$.  Since $F(G)$ is forward invariant under $G$, Pick's Theorem implies that no $g \in G$ can expand the hyperbolic distance.  Hence $D$ must be forward invariant under $G$.  Hence, $K=\CC \setminus D$ is backward invariant.  Choosing $\delta$ small enough then ensures that $a \in K$.  The remaining case $\textrm{card}(E(G))=1$ is handled similarly.
\end{proof}

Note that because $J(G)$ is the smallest closed backward invariant set containing three or more points (see~\cite{HM1}), we must have $J(G) \subseteq K$.

\section{Full backward iteration algorithm}\label{SectionFull}
Call $d=d_1+\dots +d_k$, where each $d_j=\deg f_j$.


\noindent \textbf{Notation:}  For each $z \in \CC$ and $j=1, \dots, k$, the equation $f_j(w)=z$ has $d_j$ solutions (counted according to multiplicity).  Giving these solutions an arbitrary (but fixed) order $z_{1, j}, \dots, z_{d_j, j}$, and doing so for all $z \in \CC$ and all $j=1, \dots, k$, allows us to define a collection of $d$ right inverses of the maps $f_j$, i.e., $d_j$ right inverses for each $f_j$.  Specifically, we choose $g_1(z)=z_{1, 1}, \dots, g_{d_1}(z)=z_{d_1, 1}$ and $g_{d_1+1}(z)=z_{1,2}, \dots, g_{d_1+d_2}(z)=z_{d_2, 2}$, \dots, $g_{d_1+\dots+ d_{k-1}+1}(z)=z_{1,k}, \dots, g_{d}(z)=z_{d_k,k}$.

\begin{remark}\label{Arbitrary}
Because the order of the $z_{i,j}$ is arbitrary, we cannot expect the maps $g_i$ to be true inverse branches of the $f_j$.  In fact, the order of the $z_{i,j}$ can be chosen to make each $g_i$ discontinuous at every point.  Due to the stochastic methods in what follows, however, we do not require the $g_i$ to be more than the labels that they are under this construction.  However, we do note that if $z$ is a point which is not a critical value of any of the $f_j$, then the $g_i$ may be chosen to be well defined local branches in a neighborhood of $z$ and thus be analytic there.
\end{remark}


Let $a \in \CC \setminus E(G)$ be fixed and use Lemma~\ref{LemmaK} to choose a compact set $K \subseteq \CC \setminus E(G)$ that is backward invariant under $G$ and that contains $a$.  We often choose to write  $z_{i_1, i_2, \dots, i_n}=g_{i_n}\circ \dots \circ g_{i_1}(a)$.
Note that $z_{i_1, i_2, \dots, i_n}$ depends on the initial choice of $a$ though we suppress this dependence in our notation.

Given a probability vector $b=(b_1, \dots, b_k)$, i.e., each $b_j>0$ and $\sum_{j=1}^k b_j =1$, we let $\pi_b$ be the probability measure on $\{1, \dots, d\}$ given by $\pi_b(i)=\frac{b_j}{d_j}$ whenever $d_0 +\dots +d_{j-1}+1 \leq i \leq d_0 +\dots + d_j$, where we set $d_0=0$.  We let $b'=(b_1', \dots, b_k')$ be the probability vector such that each $b_j'=\frac{d_j}{d}$ noting that this gives the uniform measure $\pi_{b'}(i)=\frac1d$ for all $i=1, \dots, d$.

For $j=1, \dots, k$, let $\nu_j^{a,b}$ be the measure on $\CC$ of total mass $b_j$ given by
$$\nu_j^{a,b} = \sum_{i=d_0 +\dots +d_{j-1}+1}^{d_0 +\dots + d_j} \pi_b(i) \delta_{z_i} = \frac{b_j}{d_j} \sum_{i=d_0 +\dots +d_{j-1}+1}^{d_0 +\dots + d_j} \delta_{z_i}$$
where $\delta_z$ denotes the unit point mass measure at $z$.

We now define probability measures $\mu_n^{a,b}$ on $\CC$ as follows:

$$\mu_1^{a,b}=\nu_1^{a,b}+\dots+\nu_k^{a,b}=\sum_{i_1=1}^d \pi_b({i_1}) \delta_{z_{i_1}}
=\sum_{i_1=1}^d \pi_b({i_1}) \delta_{g_{i_1}(a)}$$
and, in general, for $n>1$
\begin{align*}
\mu_n^{a,b}
&= \sum_{i_1, i_2, \dots, i_n =1}^d \pi_b({i_1})\cdots \pi_b({i_n}) \delta_{z_{i_1, i_2, \dots, i_n}}\\
&=\sum_{i_1, i_2, \dots, i_n =1}^d \pi_b({i_1})\cdots \pi_b({i_n}) \delta_{g_{i_n}\circ \dots \circ g_{i_1}(a)}.
\end{align*}
Note that since $K$ is backward invariant under $G$, the support of each $\mu_n^{a,b}$ is in $K$.

Following~\cite{Boyd2} and~\cite{Su3}, we define, for each $j=1, \dots, k$, a bounded linear operator $T_j=T_j^b$ on the space $C(\CC)$ of continuous functions (endowed with the sup norm $\|\cdot\|_\infty$) on $\CC$ by
\begin{equation}
(T_j \phi)(z)=\int_K \phi(s) \, d\nu_j^{z,b}(s) = \frac{b_j}{d_j} \sum_{i=d_0 +\dots +d_{j-1}+1}^{d_0 +\dots + d_j}  \phi(g_i z).
\end{equation}
Hence, each operator $T_j$ is $b_j$ times the average of the values of $\phi$ evaluated at the $d_j$ preimages of $z$ under $f_j$ (and so $\|T_j\|=b_j$).  Note that although each $g_i$ may fail to be continuous, since the \textbf{set of solutions} of $f_j(w)=z$ vary continuously in $z$, the map $T_j \phi$ is continuous.

Now define the bounded linear operator $T=T^b = \sum_{j=1}^k T_j$ which gives
\begin{equation}
(T \phi)(z) = \sum_{j=1}^k (T_j\phi)(z)=\int_K \phi(s) \, d\mu_1^{z,b}(s) = \sum_{i_1=1}^d \pi_b({i_1}) \phi(g_{i_1} z),
\end{equation}
and has $\|T\|=1$.  Hence, $(T \phi)(z)$ is a weighted average of $\phi$ evaluated at all $d$ preimages of $z$ under all generators $f_j$.  Note also that all preimages under the same $f_j$ carry the same weight (namely $\frac{b_j}{d_j}$) as determined by the carefully constructed $\pi_b$.

\begin{remark}
The continuity of $(T \phi)(z)$ (which is crucial for our purposes) depends critically on the preimages under the same $f_j$ carrying the same weight (though, as we have allowed above, the preimages under different $f_j$ may have different weights).  As an example, consider the following situation for $G=\langle z^2 \rangle$.  Here $d=2$ (and $b=(1)$ is trivial) and suppose $\pi(1)=0.9$ and $\pi(2)=0.1$ are unequal weights given to the preimages under the map $z^2$.  As we may do, choose $g_1(z)=-\sqrt{z}$ and $g_2(z)=+\sqrt{z}$ for all $z \neq 1$, but set $g_1(1)=1$ and $g_2(1)=-1$.  The measure $\mu_1^z$ is then given as $\mu_1^z = 0.9 \delta_{g_1(z)} + 0.1 \delta_{g_2(z)}$ for all $z$.  However, for $\phi \in C(\CC)$, we have $(T \phi)(z)=\int \phi \, d\mu_1^z = 0.9 \phi(g_1z)+0.1 \phi (g_2z)= 0.9 \phi(-\sqrt{z})+0.1 \phi (+\sqrt{z}) \to 0.9 \phi(-1)+0.1 \phi (1)$ as $z \to 1$ with $z \neq 1$.  However, $(T \phi)(1)=\int \phi \, d\mu_1^1 = 0.9 \phi(g_1(1))+0.1 \phi (g_2(1))= 0.9 \phi(1)+0.1 \phi (-1)$, which, for a proper choice of $\phi$ shows $(T \phi)(z)$ to fail to be continuous.  From this example we can clearly see that anywhere that $g_1$ fails to be continuous poses a similar problem and thus this problem (of trying to have unequal weights for preimages under the same map) is robust since, as a branch of the inverse of $z^2$, $g_1$ must be discontinuous somewhere.
\end{remark}

Letting $\mathcal{P}(\CC)$ denote the space of probability Borel measures on $\CC$, and noting that it is a compact metric space in the topology of weak* convergence, we have that the adjoint $T^*:\mathcal{P}(\CC) \to \mathcal{P}(\CC)$ is given by
\begin{equation}
(T^* \rho)(A)=\int \mu_1^{z,b}(A)\, d\rho(z)
\end{equation}
for all Borel sets $A \subseteq \CC$.  Using the operator notation $\langle \phi, \rho \rangle = \int \phi \, d\rho$ we express the action of the adjoint as $\langle T \phi, \rho \rangle = \langle \phi, T^* \rho \rangle$.  Note that the map $\CC \to \mathcal{P}(\CC)$ given by $z \mapsto \mu_1^{z,b}$ is continuous since for a sequence $z_n \to z_0$ in $\CC$, we have $\langle \phi, \mu_1^{z_n,b} \rangle = (T\phi)(z_n) \to (T\phi)(z_0)=\langle \phi, \mu_1^{z_0,b} \rangle$ for any $\phi \in C(\CC)$, i.e., $\mu_1^{z_n,b} \to \mu_1^{z,b}$.

The work of Lyubich~\cite{Lyu} and Friere, Lopes, and Ma\~n\'{e}~\cite{FLM, Mane} on the existence of an invariant measure on the Julia set of a single rational function was generalized to semigroups by the work of Boyd (for $b=b'$ with each $d_j \geq 2$) and Sumi (for general $b$).  The result is the following theorem.

\begin{theorem}[\cite{Boyd2, Su3}]\label{FullMethod}
Let $G$ be as in Remark~\ref{AssumptionsG} and let $b=(b_1, \dots, b_k)$ be a probability vector.  Then the measures $\mu_n^{a,b}$ converge weakly to a regular Borel probability measure $\mu^{b}=\mu^{b}_G$ independently of and locally uniformly in $a \in \CC \setminus E(G)$.  Further, the closed support of $\mu^{b}$ is $J(G)$ and $T^*\mu^{b} = \mu^{b}$.
\end{theorem}

\begin{remark}
Theorem~\ref{FullMethod} provides the justification for the ``full backward iteration algorithm" used to graphically approximate $J(G)$.  This method simply plots the $d^n$ (not necessarily distinct) points in the support of $\mu_n^{a,b}$.  We note that this iterative process plots all $d$ inverses of each of the $d^{n-1}$ points in the support of $\mu_{n-1}^{a,b}$ to generate the support of $\mu_n^{a,b}$.  Also, note that the support of $\mu_n^{a,b}$ is independent of $b$ (but not of $a$), which merely adjusts the weights on the point masses.
\end{remark}

Using Theorem~\ref{FullMethod}, the same argument given in Lemma 2.6 of~\cite{HawkinsTaylor} shows the following.

\begin{lemma}\label{Unique}
Let $G$ be as in Remark~\ref{AssumptionsG} and let $b=(b_1, \dots, b_k)$ be a probability vector.  Then the measure $\mu^b$ is the unique probability measure on $\CC$ that has support disjoint from the set of exceptional points and that is invariant under $T^*$.
\end{lemma}

\section{Random backward iteration algorithm}

Let $\Sigma_d^+ = \prod_{n=1}^\infty \{1, \dots, d\}$ denote the space of one-sided sequences on $d$ symbols, regarded with the usual topology and $\sigma$-algebra of Borel sets.   Now let
$P_b=\pi_b \times \pi_b \times \pi_b \times \dots$ be the Bernoulli measure on $\Sigma_d^+$ which is then given on basis elements as follows:  for fixed $j_n$ for $n=1, \dots, m$, we have $P_b(\{(i_1, i_2, \dots) \in \Sigma_d^+:i_n=j_n \textrm{ for all } n=1, \dots, m\})= \pi_b(j_1) \cdots \pi_b(j_m)$.

%
%

Let us define a random walk as follows.  Starting at a point $z_0=a \in \CC$, we note that there are $d$ preimages (counting according to multiplicity) of $a$ under the generators of the semigroup $G$.  We randomly select $z_1$ to be one of these preimages by a two step process, selecting a map $f_j$ with probability $b_j$ and then selecting one preimage under that map $f_j$ with (necessarily uniform) probability $\frac{1}{d_j}$.  Note that this process is the same as randomly selecting $i_1$ from $\{1, \dots, d\}$ with probability $\pi_b(i_1)$ and setting $z_1=g_{i_1}(z_0)$.  Likewise, $z_2$ is randomly selected to be one of the $d$ preimages of $z_1$.  Continue in this fashion to generate what we call a \textit{random backward orbit} $\{z_n\}$ of $z_0=a$ under the semigroup $G$.

Utilizing the notation introduced in Section~\ref{SectionFull}, we see that $\Sigma_d^+$ generates the entire set of backward orbits starting at $z_0=a$ by letting the sequence $(i_1, i_2, \dots)$ generate the backward orbit $\{g_{i_n}\circ \dots \circ g_{i_1}(z_0)\}=\{z_{i_1, i_2, \dots, i_n}\}$.  We note that the map $(i_1, i_2, \dots) \mapsto \{g_{i_n}\circ \dots \circ g_{i_1}(z_0)\}$  might not be one-to-one since more than one sequence could generate the same backward orbit.

Formally, we define our random walk $\{z_n\}$ in terms of random variables $\{Z_n\}$ given as follows.  For each $n \in \N$, we let $Z_n=Z_n^b:(\Sigma_d^+,P_b) \to \CC$ by $Z_n(i_1, i_2, \dots) = z_{i_1, i_2, \dots, i_n}$ and $Z_0 \equiv a$.  Hence, $Z_{n+1}(i_1, i_2, \dots)=g_{i_{n+1}}Z_n(i_1, i_2, \dots)$.  Note that each $Z_n$ is measurable, and, in fact, continuous, despite the arbitrariness involved in labeling the $g_i$, since for any $B \subseteq \CC$, the set $\{Z_n \in B\}=\{(i_1, i_2, \dots) \in \Sigma^+_d:Z_n(i_1, i_2, \dots) \in B\}$ is a union of cylinders of the form $\{i_1\} \times \{i_2\} \times \dots \times \{i_n\}\times \{1, \dots, d\} \times \{1, \dots, d\} \times \dots$.  This is immediate since the value of $Z_n$ depends only on the first $n$ coordinates of $(i_1, i_2, \dots)$.

%

\begin{claim}\label{Markov}
Let $G$ be as in Remark~\ref{AssumptionsG}, let $a \in \CC \setminus E(G)$, and let $b=(b_1, \dots, b_k)$ be a probability vector.  Then the stochastic process $\{Z_n:n=0, 1, 2, \dots\}$ forms a Markov process with transition probabilities $\{\mu_1^{z,b}\}$, i.e., for each Borel set $B \subseteq \CC$ we have $P_b(\{Z_{n+1} \in B|Z_0, \dots, Z_n\}) = \mu_1^{Z_n,b}(B)$.
\end{claim}

\begin{proof}
Let $D$ be in $\mathcal{D}$, the sigma algebra generated by $Z_0, Z_1,\dots, Z_n$, and let $B \subseteq \CC$ be Borel.  We will show that
the conditional probability
$P_b(Z_{n+1} \in B|Z_0,Z_1,\dots, Z_n)= \sum_{i=1}^d \pi_b(i) 1_B (g_i Z_n)=\mu_1^{Z_n,b}(B)$ by verifying that
$$\int_D 1_{\{Z_{n+1} \in B\}}\,\,dP_b=\int_D \sum_{i=1}^d \pi_b(i) 1_B (g_i
Z_n) \,\,dP_b$$
where $\sum_{i=1}^d \pi_b(i) 1_B (g_i Z_n)=\mu_1^{Z_n,b}(B)$ is $\mathcal{D}$-measurable.

To demonstrate measurability, consider \textit{any} set $A \subseteq \CC$ and note that $\{Z_n \in A\}=\{Z_n \in Z_n(\Sigma_d^+)\cap A\}$.  Since $Z_n(\Sigma_d^+)$ is a finite set, we see that $\{Z_n(\Sigma_d^+)\cap A\}$ is finite and hence Borel.  Thus, $\{Z_n \in A\} \in \mathcal{D}$.  This allows one to quickly show that each function $1_B (g_i Z_n)$ is $\mathcal{D}$-measurable, and hence so is $ \sum_{i=1}^d \pi_b(i) 1_B (g_i Z_n)$.

For each $i=1, \dots, d$, let $D_i=\{Z_n \in g_i^{-1}B\} \cap D$ and let $C_i=\{Z_{n+1} \in B\}
\cap \{i_{n+1}=i\}$.  Note that $\{i_{n+1}=i\}$ and $D_i \in \mathcal{D}$ are independent.    From this it follows that, for $i=1, \dots, d$,
\begin{align*}
\pi_b(i) P_b(D_i)
&=\pi_b(i) \frac{P_b(D_i \cap \{i_{n+1}=i\})} {P(\{i_{n+1}=i\})} \\
&=\pi_b(i) \frac{P_b(\{Z_n \in g_i^{-1}B\} \cap D \cap \{i_{n+1}=i\})}{\pi_b(i)} \\
&= P_b(\{g_i Z_{n} \in B\} \cap
D \cap \{i_{n+1}=i\})\\
&= P_b(\{g_{i_{n+1}} Z_{n} \in B\} \cap
D \cap \{i_{n+1}=i\})\\
&= P_b(\{Z_{n+1} \in B\} \cap
D \cap \{i_{n+1}=i\})\\
&= P_b(C_i \cap D).
\end{align*}

With this and the fact that the $C_i$'s are disjoint with
$\cup_{i=1}^d C_i=\{Z_{n+1} \in B\}$, we calculate
\begin{align*}
&\int_D 1_{\{Z_{n+1} \in B\}}\,\,dP_b
=\int_D \sum_{i=1}^d 1_{C_i}\,\,dP_b
=\sum_{i=1}^d P_b(C_i \cap D)
=\sum_{i=1}^d \pi_b(i) P_b(D_i)\\
&=\sum_{i=1}^d \pi_b(i) P_b(D \cap \{Z_n \in g_i^{-1}B\})
=\sum_{i=1}^d \pi_b(i) \int_D 1_{\{Z_n \in g_i^{-1}B\}}\,\,dP_b\\
&=\int_D \sum_{i=1}^d \pi_b(i) 1_{g_i^{-1}B}(Z_n)\,\,dP_b
=\int_D \sum_{i=1}^d \pi_b(i) 1_B (g_i Z_n) \,\,dP_b.
\end{align*}
\end{proof}

The main result can now be stated in terms of the probability measures, defined for each $n \in \N$ and $(i_1, i_2, \dots) \in \Sigma_d^+$, by
\begin{equation}
\mu_{i_1, \dots, i_{n}}^{a} = \frac1n \sum_{j=1}^n  \delta_{z_{i_1, i_2, \dots, i_j}}.
\end{equation}

\begin{theorem}\label{main}
Let $G$ be as in Remark~\ref{AssumptionsG}, let $a \in \CC \setminus E(G)$, and let $b=(b_1, \dots, b_k)$ be a probability vector.  Then, for $P_b$ a.a. $(i_1, i_2, \dots) \in \Sigma_d^+$, the probability measures $\mu_{i_1, \dots, i_{n}}^{a}$ converge weakly to $\mu^{b}$ in $\mathcal{P}(\CC)$.
\end{theorem}

\begin{remark}
Theorem~\ref{main} provides the justification for the ``random backward iteration algorithm" used to graphically approximate $J(G)$.  This method simply plots, for large $n$, the $n$ points in the support of $\mu_{i_1, \dots, i_{n}}^{a}$, i.e., the points of a random backward orbit, where $(i_1, i_2, \dots) \in \Sigma_d^+$ is randomly selected according to $P_b$.  If $a \notin J(G)$, then it is often appropriate to not plot the first hundred or so points in the random backward orbit since the earlier points in the orbit might not be very close to $J(G)$.
\end{remark}


\begin{remark}
The application Julia 2.0~\cite{Julia2.0} operates when the probability vector $b$ is uniform, i.e., each $b_j = \frac1k$.  Thus, Theorem~\ref{main} provides the mathematical justification for its use (when the ``ergodic/random" setting is used).
\end{remark}

\begin{remark}
We note that due to the discontinuous nature of the $g_i$, we may have a very strong dependence on the choice of $a$ regarding the convergence of $\mu_{i_1, \dots, i_{n}}^a$ to $\mu^b$.  In particular, if $\mu_{i_1, \dots, i_{n}}^a \to \mu^b$ for a particular sequence $(i_1, i_2, \dots)$, it may very well be that for a slightly perturbed value $a'$ we have that $\mu_{i_1, \dots, i_{n}}^{a'}$ converges to a measure other than $\mu^b$ or does not converge at all.  For example, consider $G=\langle z^2\rangle$.  Thus $g_1$ and $g_2$ will, at each $z$ take on the values $+\sqrt{z}$ and $-\sqrt{z}$, and can do so \emph{arbitrarily}.  So, supposing that a particular choice $(i_1, i_2, \dots)$ has $\mu_{i_1, \dots, i_{n}}^a \to \mu^b$ for $a=1$ will in no way guarantee the convergence of $\mu_{i_1, \dots, i_{n}}^{a'}$ to $\mu^b$ for $a'=1.1$.  This is because the choices of each $g_{i_j}$ (either $+\sqrt{z}$ or $-\sqrt{z}$) for $a'$ can be made completely independently of the choices for $a$ since the corresponding random backward orbits will be completely disjoint from each other.  In particular, the choices for each $g_i$ along the backward orbit of $a'$ could always be $+\sqrt{z}$, in which case we clearly see that  $\mu_{i_1, \dots, i_{n}}^{a'} \to \delta_1$.

Our setup in this paper allows for $g_i$ that are nowhere continuous, and it is this which causes the problem.  However, this is a robust problem since there is no way to choose $g_i$ in any fashion where they will be everywhere continuous (since some $g_i$ must be right inverses of a degree two or more map) and so we can never expect that backward orbits of nearby points will always be similar, even when chosen according to the same $(i_1, i_2, \dots)$.  This is in stark contrast to the contracting iterated function system case where the strong contraction leads to all orbits along the same $(i_1, i_2, \dots)$ behaving in the same way, i.e., converging closer to each other in the sense that for any $a$ and $a'$ we have \\
$\textrm{dist}(g_{i_n}\circ \dots \circ g_{i_1}(a), g_{i_n}\circ \dots \circ g_{i_1}(a')) \to 0$ as $n \to \infty$~\cite{Elton, RSThesis}.
\end{remark}

\begin{remark}
Let $\Sigma^{a,b} \subset \Sigma_d^+$ be the set of $(i_1, i_2, \dots) \in \Sigma_d^+$ such that $\mu_{i_1, \dots, i_{n}}^a \to \mu^b$, and note that $\mu^{b_1} \neq \mu^{b_2}$ implies $\Sigma^{a,b_1} \cap \Sigma^{a,b_2} = \emptyset$.  Hence we see from Theorem~\ref{main} that $P_{b_1}$ and $P_{b_2}$ are mutually singular when $\mu^{b_1} \neq \mu^{b_2}$.  If it were the case that $b_1 \neq b_2$ implies $\mu^{b_1} \neq \mu^{b_2}$ (which we suspect is generally true for most semigroups and is certainly true for the analogous situation of the iterated function system which generates the Sierpinski triangle), then we would have another proof of the fact that $P_{b_1}$ and $P_{b_2}$ are mutually singular when ${b_1} \neq {b_2}$.  This fact, however, is already known to be true by the following argument.  Since the shift map on $\Sigma_d^+$ is always ergodic with respect to either measure $P_{b_1}$ or $P_{b_2}$ (see~\cite{Walters}, p.33), $P_{b_1}$ and $P_{b_2}$ must be singular by Theorem 6.10(iv) of~\cite{Walters}.

\end{remark}

\section{Consequences of the main result Theorem~\ref{main}}

It is well known (see~\cite{HM1}) that the characterization of $J(G)$ as the smallest closed backward invariant set containing three or more points implies that for any point $a \in \CC \setminus E(G)$, we have
$$J(G) \subseteq \overline{\bigcup_{g \in G} g^{-1}(\{a\})}.$$
Furthermore, if $a \in J(G)\setminus E(G)$, then
$$J(G) = \overline{\bigcup_{g \in G} g^{-1}(\{a\})}.$$

Using Theorem~\ref{main} together with the fact that the support of $\mu^b$ is $J(G)$ no matter what probability vector $b$ is chosen, one can quickly show the following strengthening of these results.  Nearly identical details of the argument can be found in~\cite{HawkinsTaylor} and so we omit them here.

\begin{corollary}\label{Closure}
Let $G$ be as in Remark~\ref{AssumptionsG} and let $a \in \CC \setminus E(G)$.  Let $(i_1, i_2, \dots) \in \cup_b \Sigma^{a,b}$, where the union is taken over all probability vectors $b$.  Then
$$J(G) \subseteq \overline{\bigcup_{j=0}^\infty \{z_{i_1, i_2, \dots, i_j}\}}.$$
Furthermore, if $a \in J(G)$, then
$$J(G) = \overline{\bigcup_{j=0}^\infty \{z_{i_1, i_2, \dots, i_j}\}}.$$
\end{corollary}

\begin{remark}
Meeting the conclusion of this corollary is generally what one is looking for when saying that the ``random backward iteration" method works in drawing (an approximation of) $J(G)$.  Note that this conclusion holds for $P_b$ a.a. $(i_1, i_2, \dots) \in \cup_b \Sigma_d^{+}$, regardless of the choice of $b$.
\end{remark}

Suppose $(i_1, i_2, \dots) \in \Sigma_d^+$ is chosen so that $\mu_{i_1, \dots, i_{n}}^a= \frac1{n} \sum_{j=1}^n \delta_{z_{i_1, i_2, \dots, i_j}}$ converges to $\mu^b$.  Since the support of $\mu^b$ is $J(G)$, it is clear that the points $z_{i_1, i_2, \dots, i_j}$ of the backward orbit cannot visit any compact neighborhood disjoint from $J(G)$ too often, else this would imply that the $\mu^b$ measure of such a neighborhood is positive.  In this sense, \textit{most} of the points $z_{i_1, i_2, \dots, i_j}$ approach $J(G)$.  The following result, in fact, says more is true under certain circumstances.

\begin{corollary}\label{RandomLimitSet}
Let $G$ be as in Remark~\ref{AssumptionsG} and let $a \in \CC \setminus E(G)$.  Let $(i_1, i_2, \dots) \in \cup_b \Sigma^{a,b}$, where the union is taken over all probability vectors $b$.  Suppose $J(G)$ has interior or $J(G) \setminus P(G) \neq \emptyset$, where $P(G)$ denotes the postcritical set $\overline{\cup_{g \in G} \{ \textrm{critical values of }g \}}$.  Then we have
$z_{i_1, i_2, \dots, i_j} \to J(G)$, i.e., $\textrm{dist}(z_{i_1, i_2, \dots, i_j}, J(G)) \to 0$ as $j \to \infty$.
\end{corollary}

The proof of the analogous result in~\cite{HawkinsTaylor} for single function iteration utilizes Sullivan's No Wandering Domains Theorem (see~\cite{Sullivan}) and does not require the above stated hypothesis on $J(G)$.  However, the more broad setting here of rational semigroups does not offer such a No Wandering Domains Theorem, in general.  Thus our methods below offer a new approach, which we note is also applicable to the iteration result found in~\cite{HawkinsTaylor}, in those cases where $J(f) \setminus P(f) \neq \emptyset$.  Note that $J(f)$ cannot have interior without $J(f)=\CC$, which would make the conclusion of this result trivial.

\begin{proof}
Let $(i_1, i_2, \dots) \in \Sigma_d^+$ be such that $\mu_{i_1, \dots, i_{n}}^a \to \mu^b$.

First suppose that a $z$ can be chosen from the interior of $J(G)$ and note that, unlike in the iteration case, it is still possible for $J(G)$ to not equal all of $\CC$ (see\cite{HM1}).  Letting $D$ be a small spherical disk centered at $z$ and contained in $J(G)$, we see that $\mu^b(D)>0$.  Since $\mu_{i_1, \dots, i_{n}}^a \to \mu^b$, some $z_{i_1, i_2, \dots, i_{j_0}}$ (infinitely many actually) must visit $D$ and hence lie in $J(G)$.  But since $J(G)$ is backward invariant, each $z_{i_1, i_2, \dots, i_j}$ for $j \geq j_0$ must also lie in $J(G)$ and so the conclusion holds.

Now suppose $z \in J(G) \setminus P(G)$, but there exists some $\epsilon>0$ such that $\textrm{dist}(z_{i_1, i_2, \dots, i_{j_n}}, J(G)) \geq \epsilon$ along the subsequence $j_n$.  Let $\eta >0$ be chosen so that the spherical disk $\Delta(z, \eta)$ does not meet $P(G)$.  This implies that each $g \in G$ has degree$(g)$ well-defined analytic inverses on $\Delta(z, \eta)$.  By Lemma 4.5 of~\cite{Su3}, the family $\mathcal{F}$ of all these inverses for all $g \in G$ is normal on some disk $\Delta(z, \delta')$.  By equicontinuity of $\mathcal{F}$, there exists $\delta>0$ such that diam$(h(\Delta(z, \delta))<\epsilon/2$ for all $h \in \mathcal{F}$.  Since $\mu^b(\Delta(z, \delta))>0$ and $\mu_{i_1, \dots, i_{n}}^a \to \mu^b$, there exists some $z_{i_1, i_2, \dots, i_{j_0}}$ in $\Delta(z, \delta)$.  However, each $z_{i_1, i_2, \dots, i_{j}}$, for $j>j_0$, must equal $h_{j}(z_{i_1, i_2, \dots, i_{j_0}})$ for some $h_{j} \in \mathcal{F}$ (specifically $h_{j}$ must one of the branches of the inverse of $f_{i_{j_0+1}}\circ \dots \circ f_{i_j}$).  But then $\textrm{dist}(z_{i_1, i_2, \dots, i_j}, h_{j}(z)) < \epsilon/2$.  Since by backward invariance $h_{j}(z) \in J(G)$, we have a contradiction to the assumption that $\textrm{dist}(z_{i_1, i_2, \dots, i_{j_n}}, J(G)) \geq \epsilon$ along the subsequence $j_n$.
\end{proof}

\begin{remark}
Note that the maps $h_{j} \in \mathcal{F}$ chosen in this proof were not written as $g_{i_j}\circ \dots \circ g_{i_{j_0+1}}$ since, as mentioned in Remark~\ref{Arbitrary}, the $g_i$ might fail to even be continuous.  However, given that $z \notin P(G)$, we could have chosen the $g_i$ to be such that every finite composition of the $g_i$ is analytic in the disk $\Delta(z, \eta)$.
\end{remark}

\begin{remark}
There are many examples of $G$ which satisfy (1)--(3) in Remark 1.1 such that $J(G)$ has interior but where $J(G)$ is not the Riemann sphere. For example, any finitely generated $G$ consisting of polynomials of degree two or more whose Julia set contains a superattracting fixed point of some element $g\in G$ is such an example. See~\cite{HMJ}.
\end{remark}

\begin{remark}
We note that Corollaries~\ref{Closure} and~\ref{RandomLimitSet} combine to show that for the indicated $G$, sequences $(i_1, i_2, \dots) \in \cup_b \Sigma^{a,b}$ generate a backward orbit whose $\omega$-limit set is exactly $J(G)$, and so we see that $d_H(\overline{\cup_{j \geq n} \{z_{i_n, i_2, \dots, i_{j}}\}},J(G)) \to 0$ as $n \to \infty$ where $d_H$ denotes the Hausdorff metric.
\end{remark}

\section{Furstenberg-Kifer Random walks and law of large numbers}

In this section we introduce a key result of Furstenberg and Kifer to relate the invariance of the measure $\mu^b$ to the convergence of Cesaro averages of $\phi(z_n)$ where $\{z_n\}$ is a random backward orbit and $\phi \in C(\CC)$.
As in~\cite{HawkinsTaylor}, this result is the key tool in the proof of our main theorem.

Let $M$ be a compact metric space and let $\mathcal{P}(M)$ be the space of Borel probability measures on $M$, noting that $\mathcal{P}(M)$ is a compact metric space in the topology of weak* convergence.  Suppose there exists a continuous map $M \to \mathcal{P}(M)$ assigning to each $x \in M$ a measure $\mu_x$.  The corresponding Markov operator $H:C(M) \to \R$ is given by
$$Hf(x)= \int f(y) \, d\mu_x(y).$$

Suppose the stochastic process $\{X_n:n=0, 1, 2, \dots\}$ is a Markov process corresponding to $H$, i.e., $P(\{X_{n+1} \in A|X_0, X_1, \dots, X_n\}) = \mu_{X_n}(A)$.

Given these assumptions we then have the following, which is a weaker version (but sufficient for our purposes) of Theorem 1.4 in~\cite{FurstKifer}.

\begin{theorem}[\cite{FurstKifer}]\label{FKtheorem}
Assume that there is a unique probability measure $\nu$ on $M$ that is invariant under the adjoint operator $H^*$ on $\mathcal{P}(M)$ and let $\phi \in C(M)$.  Then with probability one
$$\frac1{N+1} \sum_{n=0}^N \phi(X_n) \to \int \phi \, d\nu$$
as $N \to \infty$.
\end{theorem}

\section{Proof of Theorem~\ref{main}}

By Lemma~\ref{LemmaK} we may choose a compact set $K \subseteq \CC \setminus E(G)$ that is backward invariant under $G$.  Let $\phi \in C(\CC)$ and note that its restriction to $K$ lies in $C(K)$.  Because $K$ is backward invariant we may then apply Theorem~\ref{FKtheorem} using $M=K$, $\mu_x = \mu_1^{z,b}$, $H=T$, $\nu = \mu^b$ and $X_n = Z_n$, noting that all the hypotheses have been met by Lemma~\ref{Unique} and Claim~\ref{Markov}.  Thus we obtain a set $\Sigma_\phi \subseteq \Sigma_d^+$ with $P_b(\Sigma_\phi)=1$ such that for all $(i_1, i_2, \dots) \in \Sigma_\phi$ we have $$\langle \phi, \mu_{i_1, \dots, i_{n}}^a\rangle = \frac1{n} \sum_{j=1}^n \phi(z_{i_1, i_2, \dots, i_j}) \to \int \phi \, d\mu^b= \langle \phi, \mu^b \rangle.$$

Since the set $\Sigma_\phi$ depends on $\phi$, we require an extra step to achieve single such set to work for all maps in $C(\CC)$.  Though it is perhaps standard, and was omitted in~\cite{HawkinsTaylor}, we include it here for the sake of completeness.

Since $\CC$ is compact we know that $C(\CC)$ is separable.  Let $\{\phi_j\}$ be
dense in $C(\CC)$.  Let $\Sigma_0=\cap_{j=1}^{\infty} \Sigma_{\phi_j}$ and
note that $P_b(\Sigma_0)=1$.  Let $\psi \in C(K)$.  Let $\epsilon>0$.
Select $\phi_j$ such that $\|\psi-\phi_j\|_{\infty} < \epsilon$.  Then for all $(i_1, i_2, \dots) \in \Sigma_0$, we have
$|{1 \over n}
\sum_{j=1}^n \psi(z_{i_1, i_2, \dots, i_j})
-\int_K \psi\,d\mu^b|
\leq |{1 \over n} \sum_{j=1}^n \psi(z_{i_1, i_2, \dots, i_j})
-{1 \over n} \sum_{j=1}^n \phi_j(z_{i_1, i_2, \dots, i_j})|
+|{1 \over n} \sum_{j=1}^n \phi_j(z_{i_1, i_2, \dots, i_j})
- \int_K \phi_j\,d\mu^b|
\,\,\,\,\,\,\,\,+ |\int_K \phi_j\,d\mu^b - \int_K \psi\,d\mu^b|
< \epsilon + |{1 \over n} \sum_{j=1}^n \phi_j(z_{i_1, i_2, \dots, i_j})
- \int_K \phi_j\,d\mu^b| +\epsilon
<3\epsilon$
for large $n$.


Hence on $\Sigma_0$ we get the convergence we seek, for all $\psi \in C(K)$ and this completes the proof.

Acknowledgement. The research of the second author was partially supported by JSPS KAKENHI 24540211.

\bibliographystyle{plain}
\bibliography{kyoto}

\end{document}